\documentclass{amsart}

\title{Lipschitz Constants to Curve Complexes for Punctured Surfaces}
\author{Aaron D Valdivia}
\email{aaron.david.valdivia@gmail.com}
\address{Florida Southern College; 111 Lake Hollingsworth Drive; Lakeland, FL 33801-5698}

\usepackage{amsmath}
\usepackage{amsthm}
\usepackage{amsfonts}
\usepackage{graphicx}

\newtheorem{thm}{Theorem}

\newtheorem{lem}{Lemma}

\begin{document}
\begin{abstract}
Given a surface $S_{g,n}$ there is a map $sys:\mathcal{T}_{g,n}\rightarrow\mathcal{C}_{g,n}$ where $\mathcal{T}_{g,n}$ is the Teichm\"{u}ller space with the Teichm\"{u}ller metric, $\mathcal{C}_{g,n}$ is the curve complex with the standard metric, and 
$$d_{\mathcal{C}_{g,n}}(sys(X),sys(Y))\leq Kd_{\mathcal{T}_{g,n}}(X,Y)+C.$$  
We give asymptotic bounds for the minimal value of $K$ which we denote $K_{g,n}\asymp\frac{1}{\log(\mid\chi_{g,n}\mid)}$ for sequences of surfaces with fixed genus and sequences of surfaces where the genus is a rational multiple of the punctures.  This generalizes work of Gadre, Hironaka, Kent, and Leininger where they give the same asymptotic bounds for closed surfaces.
\end{abstract}
\maketitle

Consider a surface $S_{g,n}$ of genus $g$ with $n$ punctures and negative Euler characteristic.  Let $\mathcal{T}(S_{g,n})$ be the Teichm\"{u}ller space with the Teichm\"{u}ller metric.  Let $\mathcal{C}(S_{g,n})$ be the curve complex with each edge given unit length.  We will investigate the map $sys: \mathcal{T}(S_{g,n})\rightarrow \mathcal{C}(S_{g,n})$ defined by assigning a hyperbolic metric one of its shortest curves, or systoles.  The map $sys$ is a coarse Lipschitz map \cite{MM} meaning
$$d_{\mathcal{C}_{g,n}}(sys(X),sys(Y))\leq Kd_{\mathcal{T}_{g,n}}(X,Y)+C$$ 
for some constants $K,C>0$.  In \cite{lip} the authors give asymptotic bounds for the minimal value of $K$, we will denote the minimal constant for the surface $S_{g,n}$ by $K_{g,n}$.  They show that for closed surfaces $K_{g,0}\asymp \frac{1}{\log(g)}$.  We will show that fixing the genus and increasing puctures or if the genus is a rational multiple of the punctures similar results follow.

\begin{thm}
If the genus is fixed, $g\geq 2$ or $g=rn$ for some $r\in\mathbb{Q}$ then the constant $K_{g,n}$ has behavior
$$K_{g,n}\asymp \frac{1}{\log(\mid\chi_{g,n}\mid)}.$$
\end{thm}

This is intereseting in light of other results pretaining to the lengths of geodesics in Teichm\"{u}ller space and the curve complex.  The logarithm of the minimal dilatation is similar for closed surfaces \cite{Pen91}, punctured spheres \cite{H-K}, and when the genus and punctures are related by some rational ray \cite{valdivia}, having behavior
$$\log(\delta_{g,n})\asymp\frac{1}{\mid\chi_{g,n}\mid}$$ 
but differs when the genus $g\geq 2$ is fixed \cite{Tsai} and is 
$$\log(\delta_{g,n_i})\asymp\frac{\log{\mid\chi_{g,n}\mid}}{\mid\chi_{g,n}\mid}.$$  
For asymptotic translation in the curve complex we get a similar phenomenon, if we consider closed surfaces \cite{gadretsai} or rational rays \cite{Valdivia2} we have the minimal asymptotic translation length 
$$L_{g,n}\asymp\frac{1}{\chi_{g,n}^2}.$$  
In the case of surfaces of fixed genus $g\geq 2$ \cite{Valdivia2} we have 
$$L_{g,n_i}\asymp\frac{1}{\mid\chi_{g,n}\mid}.$$  
The differing relationship between minimal dilatation and minimal asymptotic translation length for fixed genus surfaces is more coherrent in light of Theorem 1 and suggests that the Lipschitz map is finer when the punctures are increased.  As well, the bound on minimal dilatations given by Tsai \cite{Tsai} can be used with Theorem 1 to immediately get the upper bound for asymptotic translation length given in \cite{Valdivia2}.

The remainder of the paper is organized as follows, in section 2 we will establish an upper bound for $K_{g,n}$, in section 3 we will discuss bounding examples for the lower bound and conclude the proof.

\textbf{Acknowledgements:}  I would like to thank Federica Fanoni for an informative conversation about systoles of hyperbolic surfaces and Maxime Fortier Bourque for a number of helpful conversations on the material in this article.  I would also like to thank the Centro di Ricerca Matematica Ennio De Giorgi for their hospitality during the summer of 2014 without which the conversations mentioned would not have been possible.

\section{Upper bounds for $K_{g,n}$}

Adams \cite{adamssys} and Schmutz \cite{schmutzsys} give bounds for the maximal length systole in a hyperbolic surface.
\begin{thm}
If $S_{g,n}$ is a hyperbolic surface the maximum length of a systole, $sl(S_{g,n})$, is given by one of the following inequalities:

1.  If $n=1$, $sl(S_{g,n})\leq 2\mbox{arccosh}((6g-3)\slash 2).$

2.  If $n\geq 2$, $sl(S_{g,n})\leq 2\mbox{arccosh}((12g+5n+13)\slash 2).$

3.  If $n\geq 2$, $sl(S_{g,n})\leq 4\mbox{arccosh}((6g-6+3n)\slash n).$
\end{thm}
In turn this allows us to find bounds for the shortest length collar about one of the systoles of a surface.  We note that the collar length function \cite{bus},
$$w(x)=\sinh^{-1}\left(\frac{1}{\sinh(x\slash 2)}\right),$$
is monotone decreasing.  We also note that the inequality in part three of Theorem 2 bounds the length of a systole from above by a uniform constant for any fixed $g$ with $n\rightarrow\infty$ or for any sequence where $g=rn$ for some rational number $r$.  If systole length is bounded from above then the collar length is bounded from below.  This is summed up in the following lemma.

\begin{lem}
Given a fixed genus, $g$, or a sequence where $g=rn$ for some rational number, $r$, we have the width, $w$, of a collar about a systol, $\alpha\subset S_{g,n}$, is bounded below by
$$l(\alpha)\slash N\leq w.$$
For some constant N depending only on the ray in which the sequence lies. 
\end{lem}

Another short lemma gives us a bound on the intersection number for two filling curves in a punctured surface. 

\begin{lem}
If $\alpha$ and $\beta$ are simple closed curves filling $S_{g,n}$ then their intersection number in at least $-\chi_{g,n}$.
\end{lem}

\begin{proof}
If $\alpha$ and $\beta$ fill $S_{g,n}$ then it is decomposed into $i(\alpha,\beta)$ vertices, $2i(\alpha,\beta)$ edges, $D$ discs and $P$ punctured discs.  Then $2g-2+n=2i(\alpha,\beta)-i(\alpha,\beta)-D\leq i(\alpha,\beta)$.
\end{proof}

Finally we will also need a way of estimating Teichmuller distance.

\begin{lem}\cite[Wolperts Inequality]{Wolpert}
If $X$ and $Y$ are elements of $\mathcal{T}(S_{g,n})$ and $l_X(\alpha)$ is the length of a curve $\alpha$ in the metric $X$, likewise for $Y$, then
$$l_Y(\alpha)\leq e^{d_\mathcal{T}(X,Y)}l_X(\alpha).$$
\end{lem}

The following lemmas are minor adaptations from the ones in \cite{lip} used to give the upper bound for closed surfaces.

\begin{lem}(cf \cite{lip})
If $d_{\mathcal{T}}(X,Y)\leq\log((\chi_{g,n})\slash N)$ then $d_\mathcal{C}(sys(X),sys(Y))\leq 2$.
\end{lem}

\begin{proof}
Let $\alpha,\beta$ be systoles for $X$ and $Y$ respectively and assume $l_X(\alpha)\leq l_Y(\beta)$.
We have from bounds on the width of collars, 
$$\frac{i(\alpha,\beta)l_Y(\beta)}{N}< l_Y(\alpha).$$
We also have from Wolpert's inequality,
$$l_Y(\alpha)\leq e^{\log(\chi_{g,n}\slash N)}l_X(\alpha)=\frac{\chi_{g,n}l_X(\alpha)}{N}.$$
Combining these two we get
$$i(\alpha,\beta)<\chi_{g,n}\frac{l_X(\alpha)}{l_Y(\beta)}<\chi_{g,n}.$$
Therefore $\alpha$ and $\beta$ cannnot fill the surface.
\end{proof}

The upper bound for $K_{g,n}$ then follows.  

\begin{lem}(cf \cite{lip})
Given a surface $S_{g,n}$ and any two hyperbolic structures $X$ and $Y$ we have 
$$d_{\mathcal{C}}(sys(X),sys(Y))\leq\frac{2}{\log(\chi_{g,n}\slash N)}d_{\mathcal{T}}(X,Y)+2.$$
\end{lem}

\begin{proof}
There is a non-negative integer $M$ such that $M\log(\chi_{g,n}\slash N)\leq d_{\mathcal{T}}(X,Y)<(M+1)\log(\chi_{g,n}\slash N)$.  Then consider $X=X_0,X_1\dots,X_{M+1}=Y$ in $\mathcal{T}(S_{g,n})$ such that $d_{\mathcal{T}}(X_{k-1},X_{k})\leq \log(\chi_{g,n}\slash N).$  Then by the triangle inequality the upperbound follows.  
\end{proof}

\section{Examples For Lower Bounds}
For the lower bound we will be considering pseudo-Anosov mapping classes and comparing their dilatations and asymptotic translation lengths.  The comparison gives us a lower bound as a ratio.
\begin{lem}\cite{lip}
For any pseudo-Anosov $\phi:S_{g,n}\rightarrow S_{g,n}$ we have 
$$K_{g,n}\geq\frac{l_{\mathcal{C}}(\phi)}{\log(\lambda(\phi))}.$$
\end{lem}
In \cite{Tsai} Tsai gives a series of examples $\psi_{g,i}$ for any fixed $g\geq 2$ and all $i>31$ with dilatation bounded from above and below. 
$$\frac{\log(\chi_{g,n})}{C\chi_{g,n}}\leq\log(\lambda(\psi_{g,i}))\leq\frac{C\log(\chi_{g,n})}{\chi_{g,n}}$$ 
In \cite{Valdivia2} a lower bound is established for $l_{\mathcal{C}}(S_{g,n})$ for fixed genus, $g$. 
\begin{lem}\cite{Valdivia2}
For fixed genus, $g$, we have 
$$l_{\mathcal{C}}(S_{g,n})\geq\frac{1}{\mid\chi_{g,n}\mid}.$$
\end{lem}
For rational rays we will need a set of examples each rational ray.  We denote these examples $\phi_{r_i}$ where $r$ is the slope of the ray and $i$ indexes the integral points along this ray.  Given a rational number $r=\frac{p}{q}$ such that $(p,q)=1$, consider a surface $S_{p,q+2}$.  We will consider cyclic coverings of a certain mapping class on this surface.  Focusing on the dilatation, the trick is to increase the number of Dehn twists about certain curve with the index along the rational ray to obtain the correct bounds.  First we consider the set of curves depicted in Figure 1.  The curves can be given a bicoloring and we will call the set of curves colored the same color as the curve labeled $2$ by $A$ and the rest except for the curve labeled $1$ by $B$ we simply denote the curve labeled $1$ by $c$.  It is clear that that mapping classes 
$$\tilde{\phi}_{r_i}:S_{p,q+2}\rightarrow S_{p,q+2}$$ 
$$\tilde{\phi}_{r_i}:=\tau_B\circ\tau_A^{-i}\circ\tau_c$$ 
are pseudo-Anosov, where $\tau_{x}$ is the Dehn twist about the curve or multicurve $x$.
\begin{figure}[h]
\centering\includegraphics{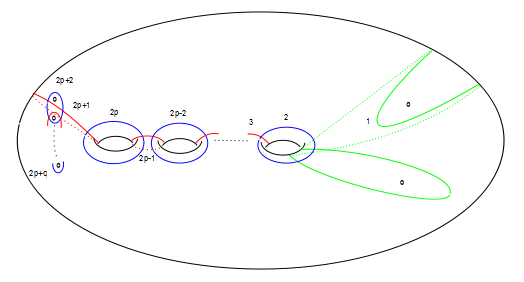}
\caption{A building block for rational rays.}
\end{figure}
For both bounds on the dilatation and asymptotic translation length we will need to use train tracks, see \cite{P-H}.  Here we will denote train tracks by $\sigma$, and the cone of measures on the train track by $P_{\sigma}$.  We will bound the dilation from above using the largest column sum of the transition matrix on this train track and we will bound the asymptotic translation length from below using the nesting lemma from \cite{MM} and an observation from \cite{gadretsai}.
\begin{lem}
If $\sigma$ is a maximal reccurrent train track for a pseudo-Anosov $\phi:S\rightarrow S$ and $r\geq 1$ such that $\phi^r(P_{\sigma})\subset int(P_{\sigma})$ then

$$l_C(\phi)\geq\frac{1}{r}.$$

The number $r$ is called a mixing number.
\end{lem}
We can find a train track for $\tilde{\phi}_{r_i}$ by smoothing intersections of curves according to the orientation of the Dehn twists being performed.  From this train track we can explicitly compute the transition matrix and find bounds for the dilitation by looking at the maximum row or column sum.  If we order the the measures of the train train $\mu_1\dots\mu_{2p+q}$ such that the the measures corresponds to the curve in the same ordering we have the following if $q$ is even, denoting $i(\alpha_x,\alpha_y)$ by $I_{x,y}$ and $2p+q=n$.  We also note that if $q$ is odd we alter only the last two rows and the column sums are similar.
$$\left(\begin{smallmatrix}
1 & 2i  &  2I_{2,3}   & 0    & 0  & 0  & 0 &.\\
I_{1,2}&I_{1,2}i+1&I_{2,3}(I_{1,2}i+1)&0&0&0&0&.\\
0&I_{2,3}i&1+I_{2,3}^2i+I_{3,4}^2i&I_{3,4}i&I_{3,4}I_{4,5}i&0&0&.\\
0&0&I_{3,4}&1&I_{4,5}&0&0&.\\
0&0&I_{3,4}I_{4,5}i&I_{4,5}i&1+I_{4,5}^2i+I_{5,6}^2i&I_{5,6}i&I_{5,6}I_{6,7}i&.\\
.&.&.&.&.&.&.&.\\
.&.&.&.&.&.&.&.\\
.&.&.&.&.&.&.&.\\
.&.&0&0&I_{n-3,n-2}I_{n-2,n-1}i&I_{n-2,n-1}i&1+I_{n-2,n-1}^2i+I_{n-1,n}^2i&I_{n-1,n}i\\
.&.&.&.&.& 0 & I_{n-1,n} & 1
\end{smallmatrix}\right)$$ 
The intersection numbers are always $1$ or $2$.  The column sums are then bounded above by $16i+9$, therefore the dilatation $\lambda(\tilde{\phi}_{r_i})\leq 16i+9$.  If we then consider cyclic coverings defined by making a cut along an arc connecting the two punctures on the right hand side of Figure 1 and fill in the punctures at those two branch points we have a surface corresponding to each integral point of the ray defined by the rational number $r=\frac{p}{q}$.  Further if we choose to lift the mapping class $\tilde{\phi}_{r_i}$ to the $i$th cyclic cover we have a mapping class on each surface whose dilatations have the same bounds.  We denote the lifted mapping classes $\phi_{r_i}$.  An example for $i=4$ and $r=\frac{2}{3}$ is depicted in Figure 2 with the lifted mapping class
$$\phi_{\frac{2}{3}_4}=\tau_{B_2}\tau_{A_2}^{-4}\tau_{c_2}\tau_{B_3}\tau_{A_3}^{-4}\tau_{c_3}\tau_{B_4}\tau_{A_4}^{-4}\tau_{c_4}\tau_{B_1}\tau_{A_1}^{-4}\tau_{c_1}.$$
Labeled in Figure 2 we have $Ai,j$, $Bi,j$, and $c,j$ we denoted $A_j=\cup_iAi,j$ and $B_j=\cup_iBi,j$.

\begin{figure}[h]
\centering\includegraphics{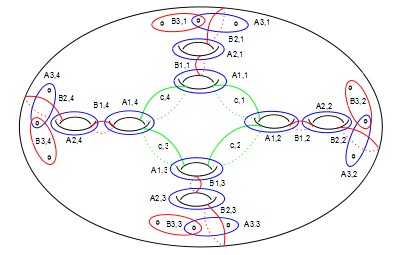}
\caption{A lifted set of curves}
\end{figure}

Further we notice that each lifted map has an $i$th root which is also psuedo-Anosov.  This mapping class can be defined as performing Dehn twists with the correct orientation to a single lift of each curve and composing with the deck transformation.  In the example in Figure 2 the 4th root of 
$$\phi_{\frac{2}{3}_4}=\rho\tau_{B_1}\tau_{A_1}^{-4}\tau_{c_1}$$
Where $\rho$ is the rotation clockwise by an angle of $\frac{\pi}{2}$ defined by the deck transformation.  We obtain the following inequality.
$$\log(\lambda(\phi_{r_i}^{\frac{1}{i}}))\leq\frac{\log(16i+9)}{i}$$
We find a mixing number for the mapping classes $\phi_{r_i}^{\frac{1}{i}}$ in order to establish the bound on asymptotic translation length.  We first notice that the image of the measure associated to the curve $A1,2$ is positive on all measures after $i$ iterations.  Then we notice that any measure will be positive on one of the measures $A1,i$ within $(2p+q)i$ iterations and after another $i$ iterations will be positive on $A1,2$.  Therefore any measure will be positive after at most $(2+2p+q)i$.  Then for any given $r=\frac{p}{q}$ we have the following.
$$l_C(\phi_{r_i}^{\frac{1}{i}})\geq\frac{1}{(2+2p+q)i}$$  
We can now finish the proof of our main theorem.

\begin{proof}
Lemma 5 gives us an upperbound for rational rays or fixed genus rays.  
$$K_{g,n}\leq \frac{1}{\log(\mid\chi_{g,n}\mid\slash N)}\asymp \frac{1}{\mid\chi_{g,n}\mid}.$$
For fixed genus rays we have 
$$\log(\lambda(\psi_i))\leq\frac{C_1\log(\mid\chi_{g,n}\mid)}{\mid\chi_{g,n}\mid}$$
for some $C_1>0$ and 
$$l_{\mathcal{C}}(\psi_i)\geq \frac{1}{C_2\mid\chi_{g,n}\mid}$$
for some $C_2>0$.  Together with Lemma 6 we have 
$$K_{g,n}\geq \frac{l_{\mathcal{C}}(\psi_i)}{\lambda(\psi_i)}\leq\frac{C_2}{C_1\log(\mid\chi_{g,n}\mid)}\asymp\frac{1}{\log(\mid\chi_{g,n}\mid)}.$$  
Together these bounds give us that for fixed genus 
$$K_{g,n}\asymp\frac{1}{\log(\mid\chi_{g,n}\mid)}.$$
For rational rays we substitute the mapping classes $\phi_{r_i}^{\frac{1}{i}}$ for the rational ray of slope $r$.

\end{proof}

\bibliography{mybib}
\bibliographystyle{alpha}

\end{document}